\documentclass[12pt,english]{amsart}

\usepackage[paperwidth=210mm,paperheight=297mm,inner=3cm,outer=3cm,top=2.5cm,bottom=2.6cm]{geometry}

\usepackage[T1]{fontenc}
\usepackage{amssymb}

\usepackage{paralist}
\usepackage{babel}
\usepackage[pdftex]{graphicx}    
\usepackage[leqno]{amsmath}
\usepackage{amsthm}
\newcommand\mbb{\mathbb}

\newcommand\mcal{\mathcal}

\newcommand\ol{\overline}

\newcommand\sC{\mcal{C}}

\newcommand\sV{\mcal{V}}

\newcommand\C{\mbb{C}}

\renewcommand\P{\mbb{P}}

\newcommand\R{\mbb{R}}
\newcommand\Z{\mbb{Z}}

\DeclareMathOperator*\Div{Div}

\DeclareMathOperator*\GL{GL}

\DeclareMathOperator*\rank{rank}

\DeclareMathOperator*\trace{tr}

\renewcommand\epsilon{\varepsilon}
\renewcommand\ge{\geqslant}

\renewcommand\le{\leqslant}

\renewcommand\phi{\varphi}

\renewcommand\theta{\vartheta}

\theoremstyle{plain}
\newtheorem{Thm}{Theorem}
\newtheorem{Prop}[Thm]{Proposition}
\newtheorem{Cor}[Thm]{Corollary}
\newtheorem{Lemma}[Thm]{Lemma}

\newtheorem*{Thm*}{Theorem}
\newtheorem*{Prop*}{Proposition}
\newtheorem*{Cor*}{Corollary}
\newtheorem*{Lemma*}{Lemma}
\newtheorem*{Sublemma*}{Sublemma}
\newtheorem*{Conjecture*}{Conjecture}

\theoremstyle{definition}
\newtheorem{Constr}[Thm]{Construction}
\newtheorem{Def}[Thm]{Definition}

\newtheorem{Example}[Thm]{Example}

\newtheorem{Remark}[Thm]{Remark}

\newtheorem*{Constr*}{Construction}
\newtheorem*{Def*}{Definition}
\newtheorem*{Defs*}{Definitions}
\newtheorem*{Example*}{Example}
\newtheorem*{Examples*}{Examples}
\newtheorem*{Exercise*}{Exercise}
\newtheorem*{LemmaDef*}{Lemma and Definition}
\newtheorem*{Notation*}{Notation}
\newtheorem*{Problem*}{Problem}
\newtheorem*{Question*}{Question}
\newtheorem*{Remark*}{Remark}
\newtheorem*{Remarks*}{Remarks}
\newtheorem*{Warning*}{Warning}

\newtheorem*{Text*}{}

\numberwithin{equation}{section}
\numberwithin{Thm}{section}

\newcommand\varx{x}

\begin{document}
\title[Determinantal representations of hyperbolic
curves]{Determinantal representations of hyperbolic plane curves: An elementary approach}

\author{Daniel Plaumann}
\address{Universit\"at Konstanz, Germany} 
\email{Daniel.Plaumann@uni-konstanz.de}

\author{Cynthia Vinzant}
\address{University of Michigan, Ann Arbor, MI, USA}
\email{vinzant@umich.edu}

\maketitle

\begin{abstract} 
If a real symmetric matrix of linear forms is positive definite at 
some point, then its determinant defines a hyperbolic hypersurface. 
In 2007, Helton and Vinnikov proved a converse in three variables, 
namely that every hyperbolic curve in the projective plane has a definite real symmetric 
determinantal representation. The goal of this paper is to give a more concrete proof of a
slightly weaker statement.  Here we show that every hyperbolic 
plane curve has a definite determinantal representation with Hermitian matrices. 
We do this by relating the definiteness of a matrix to the real topology of its minors and 
extending a construction of Dixon from 1902. Like the Helton-Vinnikov theorem, this
implies that every hyperbolic region in the plane is defined by a linear matrix inequality. 
\end{abstract}

\section{Introduction}

Let $f$ be a real homogeneous polynomial of degree $d$ in $n+1$
variables $x_0,\hdots,x_n$. A \textbf{Hermitian determinantal
  representation} of $f$ is an expression
\begin{equation}\label{eq:intro}
 f \;\;=\;\;\det(x_0M_0+\cdots+x_nM_n),
\end{equation}
where $M_0, \hdots, M_n$ are Hermitian $d\times d$ matrices. The representation is \textbf{definite} if
there is a point $e\in\R^{n+1}$ for which the matrix $e_0M_0+\cdots+e_nM_n$ is positive
definite.

The existence of a definite Hermitian determinantal representation
imposes an immediate condition on the complex hypersurface $\sV_\C(f)$
defined by $f$.
Because the eigenvalues of a Hermitian matrix are real, 
every real line passing through $e$ meets
this hypersurface in only real points.  A~polynomial with this
property is called \textbf{hyperbolic} (with respect to $e$). 
For $n=2$,
we regard $\sV_\C(f)$ as a projective plane curve. Hyperbolicity is
reflected in the topology of the real points $\sV_\R(f)$.
When the curve $\sV_{\C}(f)$ is smooth, $f$ is hyperbolic if and only if $\sV_\R(f)$ consists of $\lfloor\frac{d}{2}\rfloor$ nested
ovals, and a pseudo-line if $d$ is odd. 

\begin{figure}[h]
 \includegraphics[width=3.5cm]{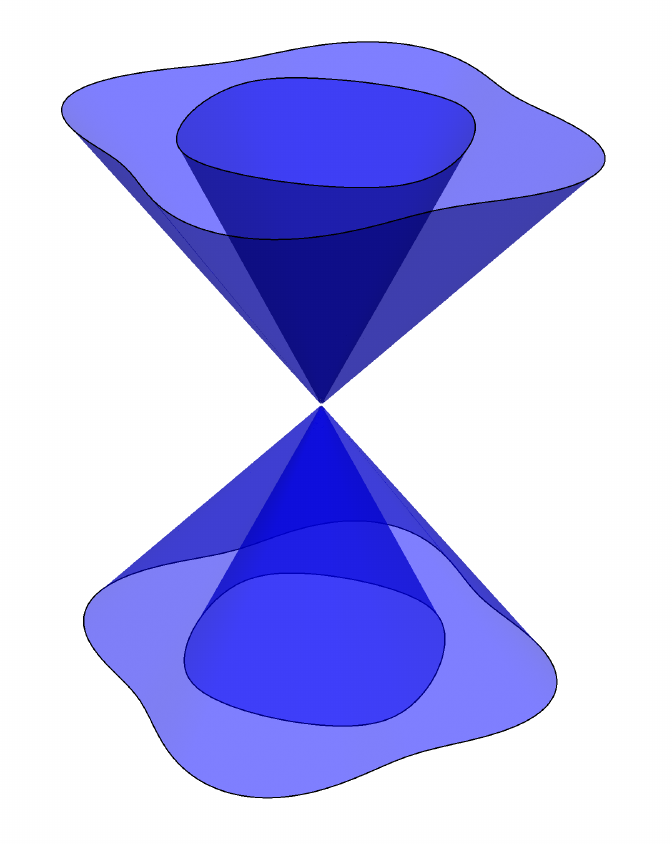} \quad \quad
  \includegraphics[width=4.0cm]{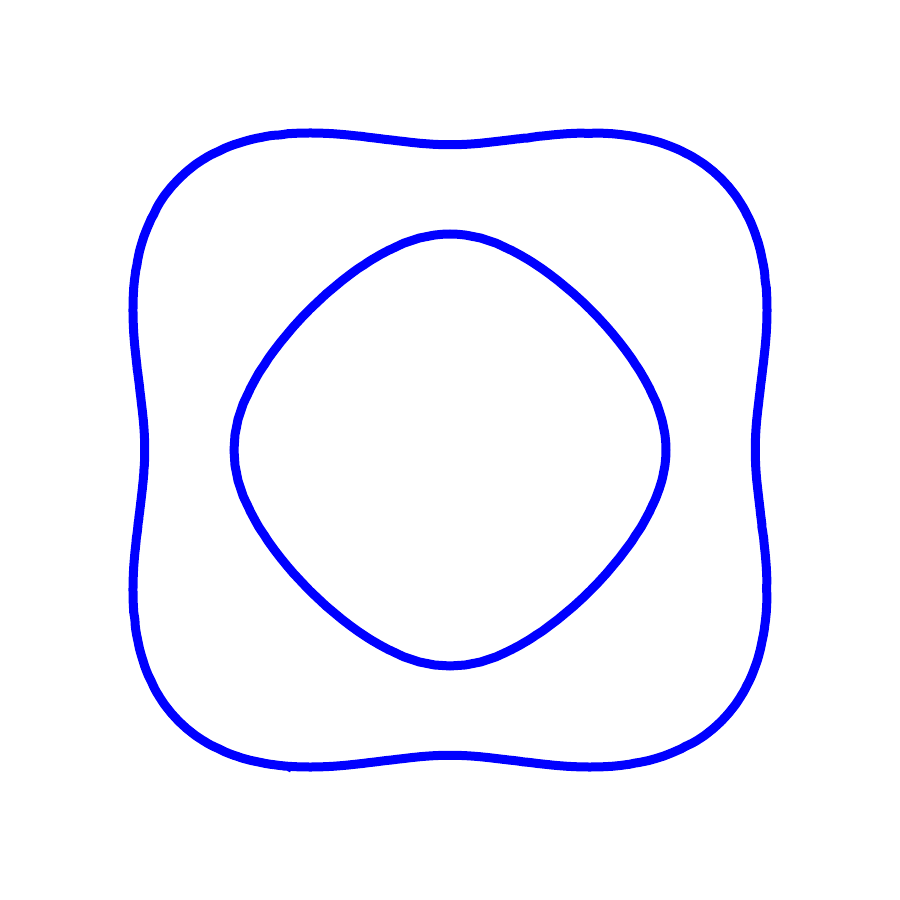}
\caption{A quartic hyperbolic hypersurface in $\R^3$ and $\P^2(\R)$.  }
\label{fig:cubic}
\end{figure}

The Helton-Vinnikov theorem \cite{HV}
(previously known as the Lax conjecture \cite{LPR}) says that for $n=2$, every
hyperbolic polynomial possesses a definite determinantal
representation \eqref{eq:intro} with real symmetric matrices. The proof is quite
involved and relies on earlier results of Vinnikov \cite{Vin} and Ball and
Vinnikov \cite{BV} on Riemann theta functions and the real structure
of the Jacobian of the curve
$\sV_{\C}(f)$. The latter had previously been studied by Gross and Harris \cite{GroHa}.

Determinantal hypersurfaces are a classical topic of complex algebraic
geometry (see Beauville \cite{Bea} and Dolgachev \cite{dolbook} for a modern presentation). In 1902,
Dixon~\cite{33.0140.04} proved that every smooth projective plane
curve admits a symmetric determinantal
representation. Hermitian and real symmetric representations of real
curves were studied in generality only later by Dubrovin \cite{MR734313} and Vinnikov
\cite{MR1024486, Vin}. 

Recently, questions in convex optimization (semidefinite and
hyperbolic programming) and operator theory have been the motivation
for more refined questions, especially concerning
the definiteness of determinantal representations. If
the polynomial $f$ has a definite determinantal representation 
$x_0M_0+x_1M_1+x_2M_2$ with real symmetric matrices, then 
the real surface defined by $f$ bounds the \emph{spectrahedron}
\begin{equation}\label{eq:spect}
\bigl\{a\in\R^3\::\: a_0M_0+a_1M_1+a_2M_2\text{ is positive semidefinite}\bigr\}.
\end{equation}
This convex set is the cone over the region enclosed by the inner oval of the
hyperbolic projective curve $\sV_\R(f)$. This realizes the convex region as the feasible set of
a semidefinite program, and we say that the region is represented by a
\emph{linear matrix inequality}. 
In this context, the Helton-Vinnikov theorem \cite{HV}
says that the convex region of any hyperbolic plane curve can be represented by a 
linear matrix inequality. 
The search for a suitable higher-dimensional analogue of this theorem is
still an intriguing open problem; 
see Vinnikov
\cite{VinReview} for an excellent review of both the history and
recent progress on this problem and Netzer and Thom \cite{NT} for further discussion.

In this paper, we give an elementary proof of the fact that every hyperbolic plane curve has a 
definite representation \eqref{eq:intro} by generalizing a classical construction 
due to Dixon \cite{33.0140.04}.  
The details of this construction and the core of the paper are in 
Section~\ref{sec:Dixon}, especially Theorem~\ref{thm:DetRepFromDixonProcess}.
Dixon's approach is to relate symmetric determinantal 
representations to families of contact curves. 
However, explicitly proving the 
existence of such curves is very difficult. 
Dixon refers to the theory of theta functions (specifically the
existence of a non-vanishing even theta characteristic). 
On the other hand, we can easily find families of curves that
correspond to Hermitian determinantal representations.
To construct a definite representation, we need only to start 
from curves that \emph{interlace} the given curve.  We can, for example, use
directional derivatives, which have
been used in the study of hyperbolicity cones already in the work of
G\r{a}rding \cite{MR0113978} and later by Renegar \cite{MR2198215},
Sanyal \cite{San}, and others.

Though we construct definite Hermitian (rather than symmetric)
matrices, the connections to convex optimization are not lost. As
discussed in Section~\ref{sec:spect}, the existence of definite Hermitian
representations still implies that the inner oval of any
hyperbolic curve is a spectrahedron.  Section~\ref{sec:hyperbolic}
contains basic facts about hyperbolicity and interlacing
polynomials. The connection between the interlacing property and the
definiteness of a Hermitian matrix of linear forms is explored in
Section~\ref{sec:HermDets}.  We provide a topological characterization
of the definiteness of such a matrix by the interlacing of its
determinant and comaximal minors (see
Theorem~\ref{thm:interlaceDetRep}).

Our overall goal is to give a new and self-contained proof of a known
result that has attracted the interest of mathematicians from many
different areas.
In this, we have tried to keep the proofs as algebraic and concrete as possible
and keep the use of topology and abstract algebraic geometry to a minimum.

\medskip
\textbf{Acknowledgements.} We would like to thank Victor Vinnikov,
David Speyer and Bernd Sturmfels for many helpful discussions. 
Daniel Plaumann was partially supported by the
research initiative \emph{Real Algebraic Geometry and Emerging
  Applications} at the University of Konstanz. Cynthia Vinzant was
partially supported by the National Science Foundation RTG grant number DMS 0943832.

\section{Hyperbolic polynomials and interlacers}\label{sec:hyperbolic}

Here we introduce the notions of hyperbolicity and interlacing
and build up some useful facts about these properties.

\begin{Def}
  A homogeneous polynomial $f\in \R[\varx]_d$ in variables
  $x=(x_0,\dots,x_n)$ is called \textbf{hyperbolic} with respect to a point
  $e\in \R^{n+1}$ if $f(e)\neq 0$ and for every $a\in \R^{n+1}$, all
  roots of the univariate polynomial $f(te+a)\in \R[t]$ are
  real. Note that since $f$ is homogeneous, this is equivalent to $f(e+ta)$ having only real roots. 
  \end{Def}

\begin{Def}
Suppose $f$ and $g$ are univariate polynomials of degrees $d$ and $d-1$ (respectively) with only real zeros. 
Denote the roots of $f$ by  $\alpha_1\le\cdots\le\alpha_d$ and the roots of $g$ by
  $\beta_1\le\cdots\le\beta_{d-1}$. We say that
  \textbf{$g$ interlaces $f$} if $\alpha_i\le\beta_i\le\alpha_{i+1}$ for all
  $i=1,\dots,d-1$. 
For multivariate polynomials, if $f \in \R[\varx]$ is hyperbolic with respect to $e\in \R^{n+1}$ and $g$ is
  homogeneous of degree $\deg(f)-1$, we say that $g$ \textbf{interlaces
    $f$ with respect to $e$} if $g(te+a)$ interlaces $f(te+a)$ in $\R[t]$ for every
  $a\in\R^{n+1}$. Note that this implies that $g$, too, is hyperbolic
  with respect to $e$. (See Figure~\ref{fig:cubic}).
\end{Def}

The most natural examples of interlacing polynomials come from derivatives. 
If $f(t)$ is a real univariate polynomial with only real roots, 
then all the roots of its derivative $f'(t)$ are real and interlace the roots of $f$. 
This easily extends to a multivariate polynomial $f(\varx)$ that is hyperbolic with respect to a point $e$. 
Since the roots of $\frac{\partial}{\partial t}f(te+a)$ interlace
those of $f(te+a)$ for all $a\in \R^{n+1}$, we see that 
\begin{equation}\label{eq:RenDeriv}
 D_ef \;\; = \;\; \sum_{i=0}^n e_i \frac{\partial f}{\partial x_i}
 \end{equation}
interlaces $f$. This was first noted by G\r{a}rding \cite{MR0113978} and used extensively in \cite{MR2198215}. 
 \begin{figure}[h]
 \includegraphics[height=4.7cm]{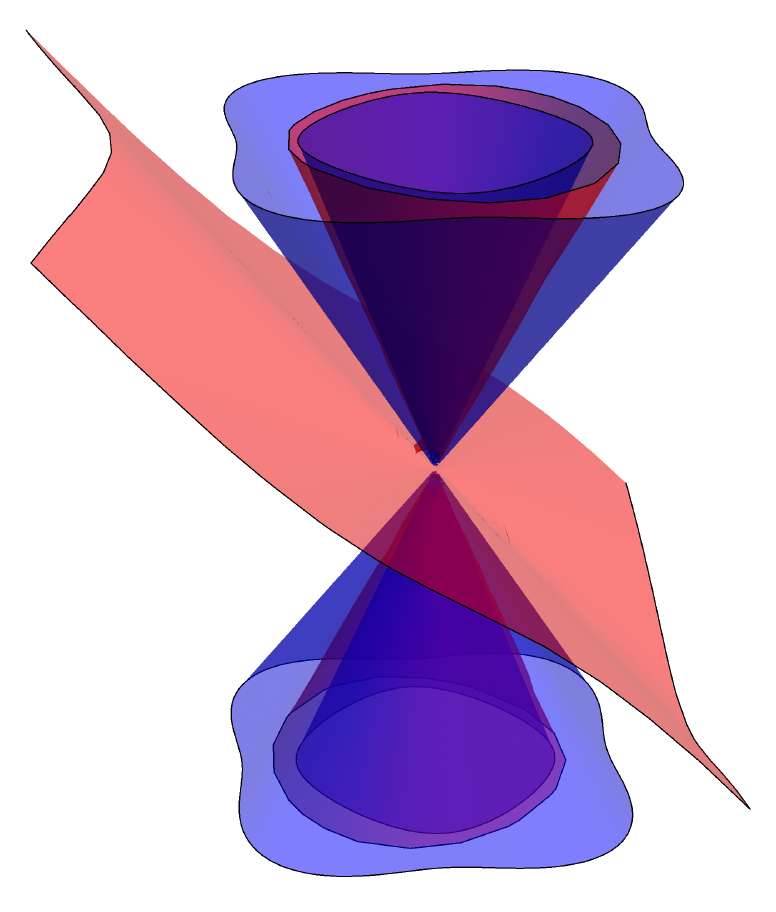} \quad \quad \quad
  \includegraphics[height=4.7cm]{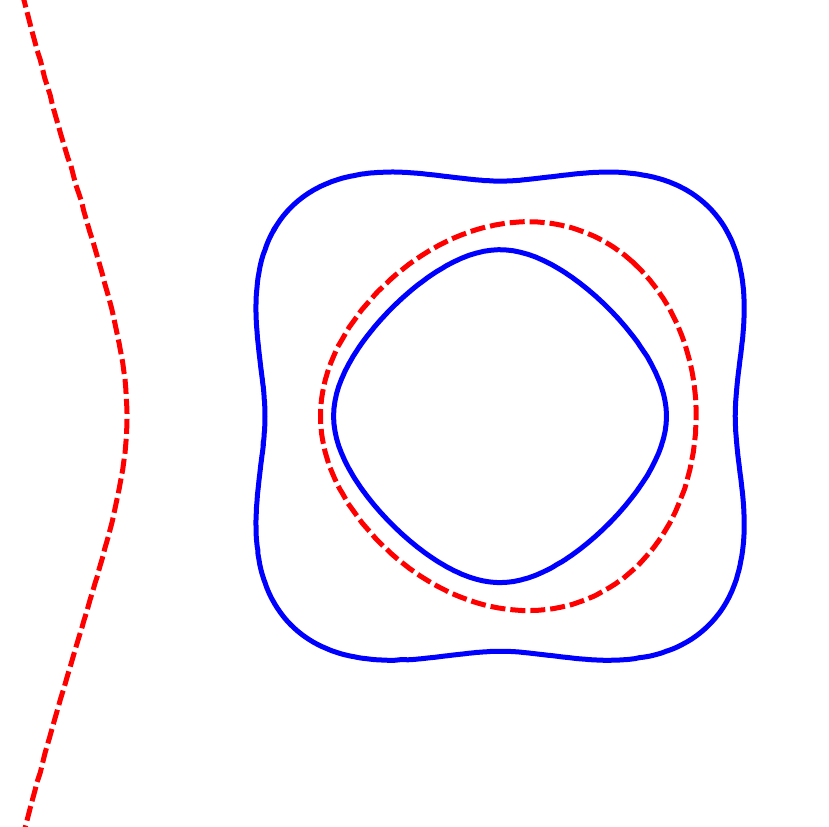} 
\caption{A cubic interlacing a quartic in $\R^3$ and $\P^2(\R)$. }
\label{fig:cubic}
\end{figure}

Note that if $f$ and $g$ are coprime and $g$ interlaces $f$ with
respect to $e$, then the roots of $f(te+a)$ are distinct from the
roots of $g(te+a)$ for points $a$ in an open dense subset of
$\R^{n+1}$. In particular, this is true of $g=D_ef$ when $f$ is
square-free.

We now come to two useful results on interlacing polynomials. The first characterizes the polynomials
that interlace $f$ by a non-negativity condition. The second characterizes the intersection 
points of $\sV_{\R}(f)$ and its directional derivative $\sV_{\R}(D_ef)$.

\begin{Lemma} \label{lem:nonnegInterlacers}
Suppose that $f \in \R[\varx]_d$ is irreducible and hyperbolic with
respect to $e$. Fix $g, h$ in $ \R[\varx]_{d-1}$ where $g$ interlaces $f$ with respect to $e$. 
Then $h$ interlaces $f$ with respect to $e$ if and only if 
$g\cdot h$ 
is nonnegative  on $\sV_{\R}(f)$ or nonpositive on $\sV_{\R}(f)$. 
\end{Lemma}

\begin{proof} 
To prove this statement, it suffices to restrict to the line $x = te+a$ for generic $a\in \R^{n+1}$. 
In particular, we may assume that the roots of $f(te+a)$ are
 distinct from each other and from the roots of $g(te+a)\cdot h(te+a)$. 
 
Suppose that $g\cdot h$ is nonnegative on $\mathcal{V}_\R(f)$. 
  By the genericity assumption, the product $g(te+a)h(te+a)$ is positive on all of
  the roots of $f(te+a)$. Between consecutive roots of $f(te+a)$, the polynomial
  $g(te+a)$ has a single root and thus changes sign. 
  For the product $g\cdot h$ to be positive on these roots, 
 $h(te+a)$ must also change sign and have a root between
  each pair of consecutive roots of $f(te+a)$. Hence $h$ interlaces
  $f$ with respect to $e$.
  
Conversely, suppose that $g$ and $h$ both interlace $f$. Between any consecutive roots of $f(te+a)$,
both $g(te+a)$ and $h(te+a)$ each have exactly one root, and their product has exactly two. 
It follows that $g(te+a)h(te+a)$ has the same sign 
on all the roots of $f(te+a)$. Taking $t \rightarrow \infty$ shows this sign to be the sign 
of $g(e)h(e)$, independent of the choice of $a$. Hence $g\cdot h$ has the same sign on every point of $\sV_{\R}(f)$.   
\end{proof}

 \begin{Lemma}\label{lemma:NonTouchingOvals}
   Let $f \in \R[\varx]$ be hyperbolic with respect to $e$. Every real intersection
   point of $\sV_{\C}(f)$ and $\sV_{\C}(D_ef)$ is a singular point of $\sV_{\C}(f)$.
 \end{Lemma}

 \begin{proof}
Suppose for the sake of contradiction that 
 a point $p$ lies in $\sV_{\R}(f)\cap \sV_{\R}(D_ef)$ and that $p$ is nonsingular in $\sV_{\C}(f)$. 
 Then the vector $\nabla f(p)$ is nonzero and orthogonal to $e$.  Now consider the affine plane 
 $H  = p+{\rm span}\{e,\nabla f(p)\}$. The restriction of $\sV_{\C}(f)$ to $H$ is a plane curve that is 
 still nonsingular at $p$. 
 
 For any $u\in H$ let $\alpha_1(u) \leq \hdots \leq \alpha_d(u)$ be the roots of $f(te+u)$ 
and for each $j=1,\hdots, d$, let $q_j(u)$ denote the point $\alpha_j(u)e+u$ in $\sV_{\R}(f)\cap H$. 
 Because $p$ lies in the intersection $\sV_{\R}(f)\cap \sV_{\R}(D_ef)$, the polynomial 
$f(te+p)$ has a double root at $t=0$.  So for some $k \in\{1,\hdots, d-1\}$, the points 
$q_k(p)$, $q_{k+1}(p)$, and $p$ are all equal.

 \begin{figure}[b]
 \includegraphics[width=2in]{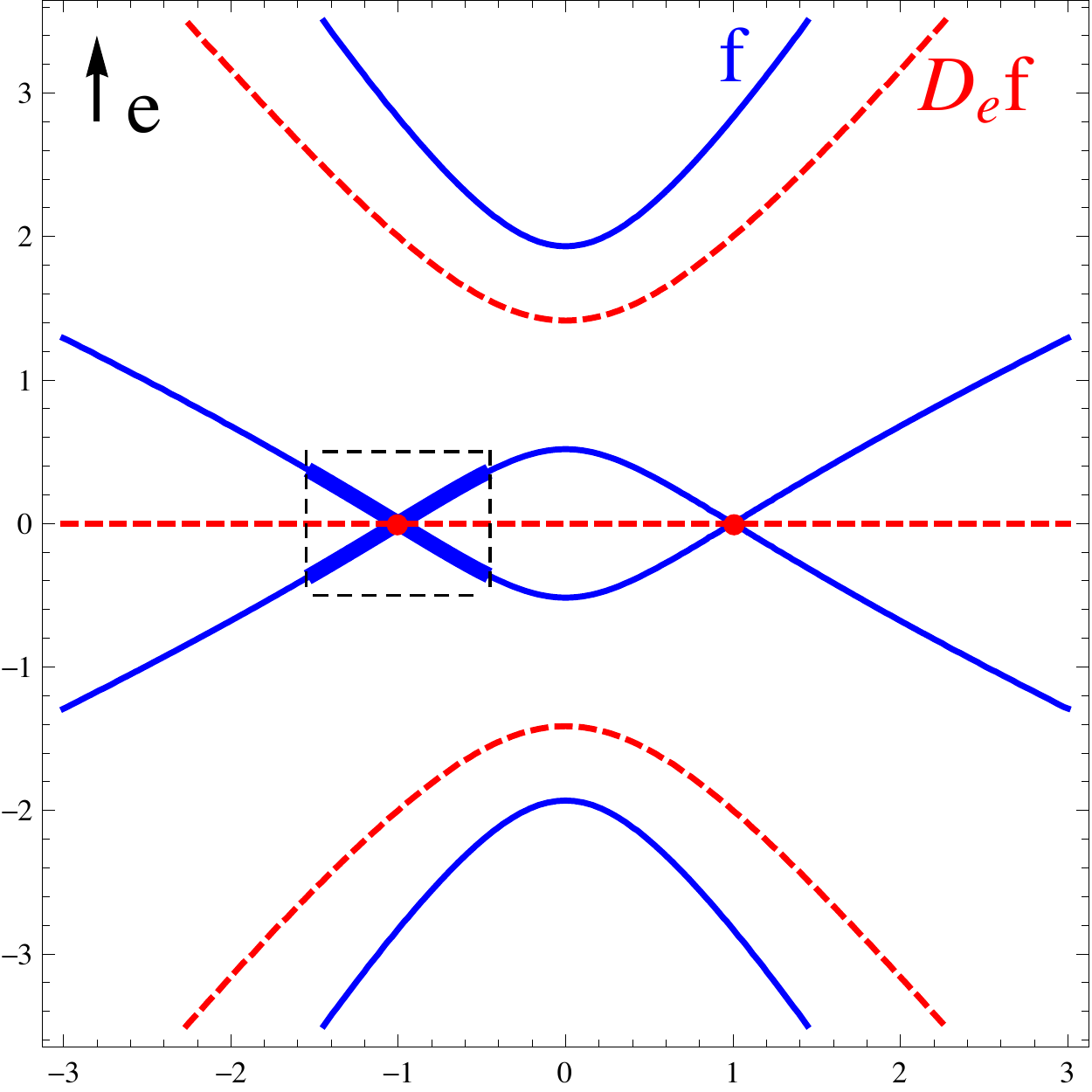} \quad \quad \quad
  \includegraphics[width=2in]{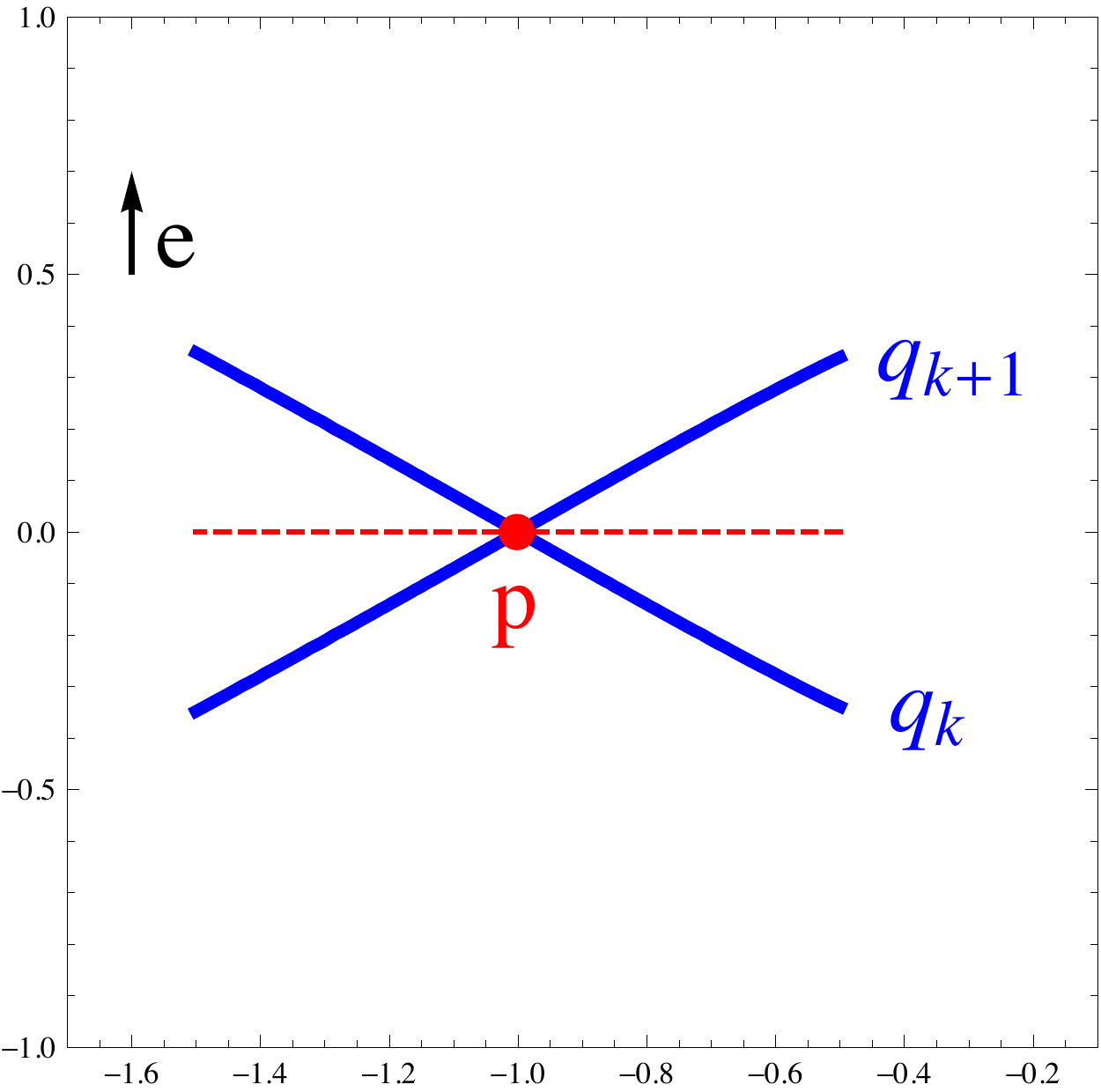} 
\caption{A singular hyperbolic curve and close up from Lemma~\ref{lemma:NonTouchingOvals}.}
\label{fig:Lemma}
\end{figure}

For all but finitely many points $u$ in the line $p+{\rm span}\{\nabla f(p)\}$, 
 the polynomial $f(te+u)$ has distinct roots. Thus we can take $U$ to be a real open 
 neighborhood of $p$ in this line such that for all $u \in U\backslash \{p\}$ the roots 
 of $f(te+u)$ are distinct. Then the maps $u \mapsto q_k(u)$ and $u\mapsto q_{k+1}(u)$ 
 give homeomorphisms between $U$ and different subsets of a neighborhood of $p$ in $\sV_{\R}(f)\cap H$.

Since $p$ is a nonsingular point of $\sV_\C(f)\cap H$,
it has an open neighborhood in $\sV_{\R}(f)\cap H$ that is homeomorphic to a line segment, 
by the implicit function theorem. 
However, by the above argument, removing the point 
$p$ from this open neighborhood results in 
at least four connected components: two in 
 $q_k(U\backslash p)$ and two in $q_{k+1}(U \backslash p)$.

This contradiction shows that $p$ must be a singular point of $\sV_{\C}(f)$. So every real intersection point of
the varieties of $f$ and $D_ef$ lies in the singular locus of $\sV_{\C}(f)$. 
 \end{proof}

\section{Interlacing and Definiteness}\label{sec:HermDets}

All eigenvalues of a Hermitian matrix are real. 
On the space of Hermitian matrices, we consider $\det(X+iY)$ as a polynomial in $\R[X_{ij}, Y_{ij}:i\leq j \in [d]]$, 
where $X = (X_{ij})$ and $Y=(Y_{ij})$ are symmetric and skew-symmetric matrices of variables, respectively.
This polynomial is hyperbolic with respect to the identity matrix.
In fact, it is hyperbolic with respect to any positive definite matrix. 
Hence, for any positive semidefinite matrix $E\neq 0$, the polynomial 
\begin{equation}\label{eq:RenMatrix}
D_E(\det(X+iY)) \;\; = \;\; \trace\!\left(E\cdot (X+iY)^{{\rm adj}}\right)
\end{equation}
interlaces $\det(X+iY)$. 
This holds true when we restrict to linear subspaces. 
For Hermitian $d\times d$  matrices $M_0, \hdots, M_n$ and variables $x=(x_0,\hdots, x_n)$, denote
\[ M(\varx) \;=\; \sum_{j=0}^n x_j M_j.\]
If $M(e)$ is positive definite for some $e\in \R^{n+1}$, 
then the polynomial $ \det(M(x))$ is hyperbolic with respect to the point $e$. 
If $E$ as above has rank one, say $E=\ol{\lambda}\lambda^T$ where $\lambda \in \C^d$,
and we restrict to the subspace of Hermitian matrices spanned by $M_0, \hdots, M_n$, 
then the polynomial \eqref{eq:RenMatrix} has the form 
$\lambda^T M^{\rm adj} \ol{\lambda}$.
In Section~\ref{sec:Dixon} we use these polynomials to \textit{reconstruct} the matrix $M$.

\begin{Def}\label{def:CM}
  Let $M$ be a $d\times d$ Hermitian matrix of linear forms. With it we associate 
 a family of polynomials,  
  \[
\sC(M)\;=\;\bigl\{\lambda^T M^{\rm adj} \ol\lambda\;\;|\;\;\lambda\in\C^d\setminus\{0\}\bigr\},
\]
and call $\sC(M)$ the \textbf{system of hypersurfaces associated with $M$}.
\end{Def}

\noindent Here is a useful identity on these hypersurfaces that 
goes back to the work of Hesse in 1855 \cite{HesseDets}.
For completeness, we include the short proof. 
\begin{Prop} \label{prop:Hesse}
Let $M$ be a Hermitian matrix of linear forms. For any  $\lambda, \mu \in \C^d$, 
\begin{equation} \label{eq:Hesse}
(\lambda^T M^{\rm adj}\overline{\lambda})(\mu^T M^{\rm adj}\overline{\mu})
\;-\; (\lambda^T M^{\rm adj}\overline{\mu})(\mu^T M^{\rm adj}\overline{\lambda}) 
\end{equation}
is contained in the ideal $( \det(M) )$.  In particular, 
the polynomial $(\lambda^T M^{\rm adj}\overline{\lambda})(\mu^T M^{\rm adj}\overline{\mu})$
is nonnegative on the real variety of $\det(M)$.
\end{Prop}

\begin{proof}
Consider a general square matrix of variables $X = (X_{ij})$. 
At a generic point in $\sV_{\C}(\det(X))$, the matrix $X$ has 
corank one.  The identity $X \cdot X^{\rm adj} = \det(X)I$
implies that $X^{\rm adj}$ has rank one at such a point. 
In particular, the $2\times 2$ matrix
$\begin{pmatrix} \lambda & \mu \end{pmatrix}^T X^{\rm adj} \begin{pmatrix} \ol\lambda & \ol\mu \end{pmatrix}$
has rank at most one on $\sV_{\C}(\det(X))$. 
Since the polynomial $\det(X)$ is irreducible, the determinant of this $2\times 2$ matrix 
lies in the ideal $( \det(X) )$.
Restricting to $X=M$ gives the desired identity. 

For the claim of nonnegativity, note that 
$(\mu^T M^{\rm adj}\overline{\lambda}) = \overline{(\lambda^T M^{\rm adj}\overline{\mu})}$. 
So the polynomial $(\lambda^T M^{\rm adj}\overline{\lambda})(\mu^T M^{\rm adj}\overline{\mu})$ 
is equal to a polynomial times its conjugate modulo the ideal $( \det(M))$. 
This shows it to be nonnegative on $\sV_{\R}(\det(M))$. 
\end{proof}

This simple identity allows us to determine whether or not a determinantal representation $M$
is definite by examining the real topology of the polynomials in $\sC(M)$. 

\begin{Thm}\label{thm:interlaceDetRep}
    Let $f\in\R[\varx]_d$ be irreducible and hyperbolic with respect to $e$,
    with a Hermitian determinantal representation $f=\det(M)$. 
  The following are equivalent:
  \begin{enumerate}
  \item Some polynomial in $\sC(M)$  interlaces $f$ with respect to $e$.
    \item Every polynomial in $\sC(M)$ interlaces $f$ with respect to $e$.
    \item The matrix $M(e)$ is (positive or negative) definite.
 \end{enumerate}
\end{Thm}

\begin{proof}(2)$\Rightarrow$(1): Clear. \smallskip

  \noindent (1)$\Rightarrow$(2): Suppose that
  $g=\lambda^TM^{\rm adj}\ol\lambda$ interlaces $f$. 
  Let $h$ denote another element of $\sC(M)$, say $h=\mu^TM^{\rm adj}\ol\mu$ where
 $\mu\in\C^d$. From Proposition~\ref{prop:Hesse}, we see that 
 the product $g\cdot h$ is nonnegative on $\sV_\R(f)$. 
 Then, by Lemma~\ref{lem:nonnegInterlacers},
  $h$ interlaces $f$. \medskip

\noindent (3)$\Rightarrow$(2): 
By switching $M$ with $-M$ and $f$ with $-f$ if necessary, 
we can take $M(e)$ to be positive definite and write it as 
$M(e) =\sum_j \overline{\mu_j}\mu_j^T$ where $\mu_j \in \C^d$. 
Then the derivative $D_e(\det(M))$ equals 
\[D_e(\det(M)) \; =\; \trace\!\left(M(e) \cdot M^{\rm adj}\right) \; =\; 
 \trace\!\left( \sum_j \overline{\mu_j}\mu_j^T\cdot  M^{\rm adj}\right)
\;=\; \sum_j  \mu_j^TM^{\rm adj}\overline{\mu_j}.\]
Then for any $\lambda \in \C^d$, the polynomial 
\[ (\lambda^TM^{\rm adj} \overline{\lambda}) \cdot D_e(\det(M))
\;\;=\;\; \sum_{j} (\lambda^TM^{\rm adj} \overline{\lambda})(\mu_j^TM^{\rm adj}\overline{\mu_j})
\]
is nonnegative on $\sV_{\R}(\det(M))$, using Proposition~\ref{prop:Hesse}.  
Because $D_e(\det(M))$ interlaces
$\det(M)$ with respect to $e$, we can then use Lemma~\ref{lem:nonnegInterlacers} to see that 
$\lambda^TM^{\rm adj} \overline{\lambda}$ also interlaces $\det(M)$ with respect to $e$. \medskip

\noindent (2)$\Rightarrow$(3): 
First, let us show that any two elements $g, h$ of $\sC(M)$ have the same sign at the point $e$. 
Since $f$ is irreducible, the polynomial $g\cdot h$
cannot vanish on $\sV_{\R}(f)$. 
By Proposition~\ref{prop:Hesse}, the product $g\cdot h$ is nonnegative on $\mathcal{V}_{\R}(f)$ and 
thus strictly positive on a dense subset of $\mathcal{V}_{\R}(f)$.  Furthermore, because both $g$ and $h$ interlace 
$f$, they cannot have any zeroes in the component of $e$ in $\R^{n+1}\backslash \sV_{\R}(f)$.  
So the product $g\cdot h$ must be positive on this component of $e$ in $\R^{n+1}\backslash \sV_{\R}(f)$ and thus at $e$ itself.  

Now consider the Hermitian matrix $M^{\rm adj}(e)$. 
We have shown that the sign of $ \lambda^T M^{\rm adj}(e)\overline{\lambda}$
is the same for every $\lambda \in \C^d$. 
This shows that the matrix $M^{\rm adj}(e)$ is definite, hence so is $M(e) = f(e)(M^{\rm adj}(e))^{-1}$.
\end{proof}
The diagonal $(d-1)\times(d-1)$ minors of $M$ are elements of
$\sC(M)$.  So a corollary of Theorem~\ref{thm:interlaceDetRep} is that
a linear subspace of Hermitian matrices contains a definite matrix if
and only if its diagonal co-maximal minors interlace its
determinant. For an alternative proof of this fact, see \cite[Theorem~5.3]{VinReview}.

We conclude this section with a useful lemma about limits of
determinantal representations: The map taking a matrix with
linear entries to the determinant is closed when restricted to
definite representations, which it need not be in general.
This was also shown by Speyer \cite[Lemma 8]{Speyer}.

\begin{Lemma}\label{lemma:DefDetClosed}
  Let $e\in\R^{n+1}$. The set of homogeneous polynomials
  $f\in\R[x]_d$ with $f(e)=1$ that possess a Hermitian determinantal
  representation $f=\det(M)$ where $M(e)$ is positive definite
  is closed in $\R[x]_d$. 
\end{Lemma}

\begin{proof}
First we observe that if $f(e)=1$ and $f=\det(M)$ where $M(e)\succ 0$, 
then $f$ has such a representation $M'$ for which $M'(e)$ is the identity matrix. 
To find it, we can decompose the matrix $M(e)^{-1}$ as $\ol{U}U^T$ where $U\in \C^{d\times d}$.
Then $M'=U^TM\ol{U}$ is a definite determinantal representation of $f$ with $M'(e)=I$.

  Now let $f_k\in\R[x]_d$ be a sequence of polynomials converging to $f$ 
  such that $f_k=\det(M^{(k)})$ with $M^{(k)}(x)=x_0M^{(k)}_0+\cdots+x_nM^{(k)}_n$
  and $M^{(k)}(e)=I_d$. For each $j$, let $e_j$ denote the $j$th unit vector. 
  Since  $f_k(te-e_j)$ is the characteristic polynomial of $M^{(k)}_j$, 
  the eigenvalues of each $M^{(k)}_j$ converge to the zeros
  of $f(te-e_j)$. It follows that each sequence $(M^{(k)}_j)_k$ is
  bounded. We may therefore assume that the sequence $M^{(k)}$ is
  convergent (after successively passing to a convergent subsequence
  of $M^{(k)}_j$ for each $j=0,\dots,n$) and conclude that
  $f=\det(\lim_{k\to\infty} M^{(k)})$.
\end{proof}

\section{Dixon's construction for hyperbolic curves}\label{sec:Dixon}

Here we describe a modification of the classical construction of
Dixon~\cite{33.0140.04}, which relates determinantal representations
of plane curves to contact curves.  Dixon considered only determinants
of symmetric matrices.  As described below, we can use a similar
method to construct Hermitian determinantal representations.  The
exact relation of these determinantal representations to families of
``contact'' curves is somewhat subtle and has been worked out by
Vinnikov in \cite{MR1024486}.  Here we give an account using 
only intersection theory of plane curves
(all of which can be found for example in \cite{MR1042981}), and refer
to \cite{MR1024486} for more detailed information.  Because we now
deal only with plane curves, we fix $n=2$ and replace 
$(x_0,x_1,x_2)$ by $(x,y,z)$.

As we saw in Proposition~\ref{prop:Hesse}, for any square matrix $M$ of linear forms, 
the matrix $M^{\rm adj}$ has rank at most one along $\sV_{\C}(\det(M))$. In particular, 
its $2\times 2$ minors lie in the ideal generated by $\det(M)$. 
The main idea of Dixon is to reconstruct $M$ by producing a suitable $A = M^{\rm adj}$,
namely, a $d\times d$ matrix of forms of degree $d-1$ whose $2 \times 2$ minors lie in the ideal $(\det(M))$. 
We modify his construction to produce a Hermitian determinantal representation. 
Theorem~\ref{thm:interlaceDetRep} shows that if the top left entry of $A$ interlaces $f$, then 
this determinantal representation will be definite. 

Here is a summary of the construction.  
The input is a smooth real form $f$ that is hyperbolic with respect to a point $e=(e_0,e_1,e_2)\in \R^3$, and the output
is a definite Hermitian determinantal representation $M$ of $f$. \medskip

\begin{itemize}
\item Let $a_{11} $ be the form 
$D_ef = e_0\frac{\partial f}{\partial x}+e_1\frac{\partial f}{\partial
  y}+e_2\frac{\partial f}{\partial z}$ of degree $(d-1)$. \smallskip 
\item Split the $d(d-1)$ points $\sV_{\C}(f)\cap \sV_{\C}(a_{11})$ into two disjoint, conjugate sets of points $S \cup \overline{S}$.  \smallskip
\item Extend $a_{11}$ to a basis $\{a_{11}, \hdots, a_{1d}\}$ of the forms in $\C[x,y,z]$ 
of degree $d-1$ that vanish on the set of points $S$. \smallskip
\item For $1<j\leq k$, let $a_{jk}$ be a polynomial for which $a_{11}a_{jk} - \overline{a_{1j}}a_{1k}$ lies in the ideal $( f )$, with $a_{jk}$ real if $j=k$.  For $j<k$, define $a_{kj} = \overline{a_{jk}}$ and define $A = (a_{jk})_{j,k}$ to be the resulting $d\times d$ matrix of forms of degree $d-1$.\smallskip
\item Define $M$ to be the matrix of linear forms obtained by dividing each entry of $A^{\rm adj}$ by $f^{d-2}$.
\end{itemize}

We will show that these steps can be carried through and that the resulting matrix $M$ is a
definite determinantal representation of $f$. 
We see that the output depends on some choices, the most important of which is the
  splitting of the points $\sV_{\C}(f,a_{11})$. The resulting determinantal representation 
  depends on the divisor of $\sV_{\C}(f)$ consisting of the points $S$.  
  To discuss this precisely, we use the language of divisors on curves. 

For a form $f\in \C[x,y,z]$, let
$\Div(f)$ denote the free abelian group over the complex points
$\sV_{\C}(f)$. Thus an element of $\Div(f)$ is an expression $D=\sum_{i=1}^k
n_iP_i$ with $P_1,\dots,P_k\in \sV_{\C}(f)$ and $n_i\in\Z$, called
a \emph{divisor on $\sV_{\C}(f)$}. The degree of the divisor $D$ is defined by
$\deg(D)=\sum_{i=1}^k n_i$ and its conjugate divisor is 
$\ol{D}=\sum_{i=1}^kn_i\ol{P_i}$.
If $g \in \C[x,y,z]$ is also homogeneous and
shares no factors with $f$, then the
\emph{intersection divisor} of $f$ and $g$ is defined by
$f.g=\sum_{P\in \sV_{\C}(f,g)}I_P(f,g)\cdot P$, where $I_P(f,g)$ is the
intersection multiplicity of $f$ and $g$ at the point $P$. 
For two forms $g, h$ in $\C[x,y,z]$ that are coprime to $f$, we have that 
$f.(gh)=f.g+f.h$ and $f.g = f.(g+hf)$.  
If $f$ has degree $d$
and $g$ has degree $e$, B\'ezout's theorem says that
$\deg(f.g)=de$. Given a divisor $D=\sum_{i=1}^k n_i P_i$, we
write $D\ge 0$ if $n_i\ge 0$ for $i=1,\dots,k$. For two divisors
$D,E\in\Div(f)$, write $E\ge D$ if $E-D\ge 0$. We need the
following classical result:

\begin{Thm}[Max Noether]\label{thm:maxnoether} Let $\sV_{\C}(f)$ be a smooth projective plane
  curve over $\C$ and let $g,h\in\C[x,y,z]$ be homogeneous. Assume
  that $g$ and $h$ have no irreducible components in common with
  $f$. If $f.h\ge f.g$, then there exist homogeneous
  polynomials $a,b\in\C[x,y,z]$ such that $h=af+bg$.
  If $f,g,h$ are all real, then $a,b$ can also be chosen to be real.
\end{Thm}

\begin{proof}
  See, for example, \cite[\S
  5.5]{MR1042981} for the proof. 
For the reality of $a$ and $b$, note that if $f,g,h$ are
  all real and $h=af+bg$, then $h=\frac12(a+\ol{a})f+\frac12(b+\ol{b})g$.
\end{proof}

The intersection divisors of interest to us come from curves that have 
special intersection with the set of real points $\sV_{\R}(f)$.

\begin{Def}
Let $f, g \in \R[x,y,z]$. Then $\sV_{\C}(g)$ is a \textbf{curve of real contact} of $\sV_{\C}(f)$
  if there exists a divisor $D\in\Div(f)$ such that $f.g=D+\ol{D}$. In
  this case, the divisor $D$ is called a \textbf{real-contact divisor} of $\sV_{\C}(f)$.
\end{Def}

In other words, a real plane curve is a curve of real contact of $\sV_{\C}(f)$ if and only
if all real intersection points with $\sV_{\C}(f)$ have even mulitplicity. The simplest example of such a 
curve $\sV_{\C}(g)$ is one for which $\sV_{\R}(f)\cap \sV_{\R}(g) = \emptyset$.  
For us, the most important examples come from real curves that interlace $f$. 

\begin{Prop}\label{prop:InterlacingRealContact}
Suppose $f \in \R[x,y,z]$ is hyperbolic with respect to $e$ and that the curve
$\sV_{\C}(f)$ has no real singular points. Then any form that interlaces $f$ with 
respect to $e$ is a curve of real contact of $\sV_{\C}(f)$.
\end{Prop}

\begin{proof}
  If $g \in \R[x,y,z]$ interlaces $f$ with respect to $e$, then 
  by Lemma~\ref{lem:nonnegInterlacers}, the product $g\cdot D_ef$ has constant
  sign on $\sV_\R(f)$. If $\sV_{\R}(f)\cap \sV_{\R}(g)$ is empty, then 
  $g$ is automatically a curve of real contact to $\sV_{\C}(f)$. 
  On the other hand, suppose there is a point $P$ in $\sV_{\R}(f)  \cap \sV_{\R}(g)$. 
  By Lemma~\ref{lemma:NonTouchingOvals},  $D_ef(P)$ is nonzero. It follows that the
  restriction of $g$ to $\sV_{\C}(f)$ has locally constant sign around
  $P$ and has therefore even vanishing order in $P$. That vanishing order
  is exactly the intersection multiplicity of $f$ and $g$ in $P$,
  meaning that every real intersection point of $\sV_{\C}(f)$ and $\sV_{\C}(g)$ has even multiplicity.
\end{proof}

One can also obtain real-contact divisors directly from a Hermitian determinantal 
representation, as shown in the following proposition. Our eventual goal 
is to reconstruct the determinantal representation from such a divisor. 

\begin{Prop}\label{prop:divM}
 Let $\sV_{\C}(f)$ be a smooth real projective plane curve of degree $d$ 
 and let $f=\det(M)$ be a Hermitian linear determinantal
 representation. Let $(a_{11},\dots,a_{1d})$ be the first row of
 $M^{\rm adj}$. For each intersection point $P\in\sV_{\C}(a_{11},f)$,
  let $n_P=\min\{I_P(f,a_{1j})\:|\: j=1,\dots,d\}$.   Then
\[
D_M=\sum_{P\in\sV_{\C}(a_{11},f)} n_P P
\]
is a real-contact divisor of degree $d(d-1)/2$ on $\sV_{\C}(f)$ with
$f.a_{11}=D_M+\ol{D_M}$.
\end{Prop}

\begin{proof}
Using Proposition~\ref{prop:Hesse}, we have that 
$a_{jj}a_{kk}-a_{jk}\ol{a_{jk}}$ lies in the ideal $(f)$
for all $j, k$, which shows the two intersection divisors
$f.(a_{jj}a_{kk})$ and $f.(a_{jk}\ol{a_{jk}})$ to be equal. 
 Let $P\in \sV_{\C}(a_{11},f)$. Since $\sV_{\C}(f)$ is
smooth, $M^{\rm adj}(P)$ is not the zero matrix and there exists 
$j$ for which $a_{jj}(P)\neq 0$.
This implies that
\[I_P(f,a_{11})=I_P(f,a_{11}a_{jj})
=I_P(f,a_{1j}\ol{a_{1j}}),\] which shows that
the multiplicity $I_P(f,a_{11})$ is even and $\sV_{\C}(a_{11})$ is a
curve of real contact of $\sV_{\C}(f)$. Furthermore, by definition we have that
$I_P(f,a_{1j}\ol{a_{1j}})\ge n_P+n_{\ol P}$, which shows that 
$f.a_{11}\ge D_M+\ol{D_M}$. 
On the other hand, 
\[I_P(f,a_{1k}\ol{a_{1k}}) =
I_P(f,a_{11}a_{kk}) \ge I_P(f,a_{11})\]
holds for any
$k\in\{1,\dots,d\}$, and thus $f.a_{11}\le D_M+\ol{D_M}$.
This shows that $a_{11}$ is a curve of real contact and that $D_M$ is a real-contact divisor. 
\end{proof}

If the matrix $M$ is real, then $D_M$ equals $\overline{D_M}$ 
and $f.a_{11}$ equals $2D_M$, which puts a strong restriction on possible choices of $a_{11}$.  
This is the original setting of Dixon's algorithm.
The following is a modification of his construction, which reconstructs 
the Hermitian determinantal representation $M$ from a real-contact divisor $D_{M}$.

\begin{Constr}[of $A = M^{\rm adj}$]\label{constr:Dixon} 
  Let $D$ be a real-contact
  divisor of degree $\binom{d}{2}$ on $\sV_{\C}(f)$. We construct a Hermitian
  matrix $A_D$ with entries in $\C[x,y,z]_{d-1}$ as follows: 
  
  Let $a_{11}\in\R[x,y,z]_{d-1}$ be such that
  $f.a_{11}=D+\ol{D}$. Consider the complex vector space
  $V$ of polynomials $g\in\C[x,y,z]_{d-1}$ for which $f.g\ge D$. The dimension of $V$
  is at least $\binom{d+1}{2}-\binom{d}{2}=d$ (the dimension of
  $\C[x,y,z]_{d-1}$ minus the maximal number $\binom{d}{2}=\deg(D)$ of
  linearly independent
  conditions imposed on $g$ by $f.g\ge D$).
  
  Extend $a_{11}$ to a linearly independent family
  $a_{11},\dots,a_{1d}$ in $V$. 
  For $2\le j\le k\le d$, 
  we have $f.(\ol{a_{1j}}a_{1k})\ge
  D+\ol{D}=f.a_{11}$. Thus we can apply
  Theorem~\ref{thm:maxnoether} and obtain homogeneous polynomials
  $p,q\in\C[x,y,z]$ such that $\ol{a_{1j}}a_{1k}=pf+qa_{11}$. 
  Put $a_{jk}=q$. If $j=k$, then $a_{11}$ and $\ol{a_{1j}}a_{1j}$ are both real and
  we can take $a_{jj}$ real as well. Finally, put $a_{k j}=\ol{a_{jk}}$
  for $j< k$ and let $A_D=(a_{jk})_{j,k}$.
\end{Constr}

\noindent We let $A_D$ denote any matrix resulting
from the above construction. This will be the adjugate matrix of a determinantal representation of $f$. 
When $f$ is hyperbolic and $a_{11}$ interlaces $f$, then the representation will be definite. 

\pagebreak

\begin{Thm}\label{thm:DetRepFromDixonProcess}
  Let $\sV_{\C}(f)$ be a smooth real projective plane curve of degree $d$. 
   Sup\nolinebreak pose $D$ is a real-contact divisor of degree $\binom{d}{2}$ of $\sV_{\C}(f)$.
  \begin{enumerate}[(a)]
  \item Every entry of the adjugate matrix of $A_D$
  is divisible by $f^{d-2}$ and the matrix $M_D=(1/f^{d-2})A_D^{\rm
    adj}$ has linear entries. Furthermore there exists $\gamma\in\R$ such that
\[
\gamma f=\det(M_D).
\]
\item If $f$ is hyperbolic and $a_{11}=(A_D)_{11}$ interlaces $f$ with respect to $e$, 
then $\gamma\neq 0$ and the matrix $M_D(e)$ is (positive or negative) definite. 
\end{enumerate}

\end{Thm}

\noindent The following lemma will be essential for the proof of this theorem.

\begin{Lemma}\label{lemma:minorsinrings}
  Let $A$ be a $d\times d$-matrix
  with entries in $\C[x,y,z]$ with $d\ge 2$. Let $f\in \C[x,y,z]$ be irreducible. 
  If $f$ divides all $2\times 2$-minors of $A$, then for every $1\le k\le d$,  
  the polynomial $f^{k-1}$ divides every $k\times k$-minor of $A$.
\end{Lemma}

\begin{proof}[Proof of Lemma]
  By hypothesis, the claim holds for $k=2$. So assume $k>2$ and
  suppose that $f^{k-2}$ divides all $(k-1)\times (k-1)$-minors of
  $A$. Let $B$ be a submatrix of size $k\times k$ of $A$. From $B^{\rm
    adj}B=\det(B)\cdot I_k$ we conclude $\det(B^{\rm
    adj})=\det(B)^{k-1}$. 
    
    Suppose $\det(B) = f^m g$ where $f$ does not divide $g$.
    Then $\det(B)^{k-1} = f^{m(k-1)}g^{k-1}$. 
    By assumption $f^{k-2}$ divides all entries of $B^{\rm
    adj}$, hence $f^{k(k-2)}$ divides its determinant $\det(B)^{k-1}$.     
      Since $f$ is irreducible, $f$ does not divide  $g^{k-1}$, so $f^{k(k-2)}$ must divide 
      $f^{m(k-1)}$. Then $k(k-2)\leq m(k-1)$ which implies that $k-1\leq m$, as claimed. 
\end{proof}

\begin{proof}[Proof of Theorem~\ref{thm:DetRepFromDixonProcess}(a)]
  By construction, the $2\times 2$ minors of $A_D$ having the form $a_{11}a_{jk}-a_{1k}a_{j1}$
  are divisible by $f$. This means that for every point $p$ in $\sV_{\C}(f)$ with $a_{11}(p)\neq 0$, 
  all rows of the matrix $A_D(p)$ are multiples of the first, and thus $A_D(p)$ has rank one.     
  Because $a_{11}$ is not divisible by $f$, it follows that 
  all the $2\times 2$ minors of $A_D$ are divisible by $f$. 
  Since $f$ is irreducible in
  $\C[x,y,z]$, all $(d-1)\times (d-1)$-minors of
  $A_D$ are divisble by $f^{d-2}$, by Lemma
  \ref{lemma:minorsinrings}. The entries of $A_D^{\rm adj}$ have degree
  $(d-1)^2$ and $f$ has degree $d$, so that $M_D=(1/f^{d-2})\cdot
  A_D^{\rm adj}$ has entries of degree $(d-1)^2-d(d-2)=1$. 
  Furthermore, by Lemma \ref{lemma:minorsinrings}, 
  $\det(A_D)$ is divisible by $f^{d-1}$. So $\det(A_D)=cf^{d-1}$
  for some $c\in\R[x,y,z]$ and we obtain
\begin{align*}
\det(M_D) &=\det(f^{2-d}A_D^{\rm adj})=f^{d(2-d)}\det(A_D^{\rm
  adj})=f^{d(2-d)}\det(A_D)^{d-1}\\
& =f^{d(2-d)}c^{d-1}f^{(d-1)^2}=c^{d-1}f.
\end{align*}
Since $\det(M_D)$ has degree $d$, we see that $c$ is a constant and we take $\gamma = c^{d-1}$.
\end{proof}

This gives us a potential determinantal representation of $f$. To finish the job, 
we need to ensure that the constant $\gamma$ is nonzero. 
Following Dixon \cite{33.0140.04}, we do this by analyzing the 
system of curves associated with $A_D$.

\begin{Lemma}\label{lem:CD}
  Let $D$ be a real-contact divisor of degree $d(d-1)/2$ of $\sV_{\C}(f)$ and write
\[
\sC_D=\bigl\{\lambda^TA_D\ol{\lambda}\:|\: \lambda\in\C^d\setminus\{0\}\bigr\}.
\]
Every element of $\sC_D$ is a curve of real contact of $\sV_{\C}(f)$ of degree $d-1$.  \end{Lemma}

\begin{proof}
By assumption, the polynomial $a_{11} = (e_1^TA_De_1)$ is a curve of real contact to $f$ 
with real-contact divisor $D$. 
By the preceding proof, the matrix $A_D$ has rank one on $\sV_{\C}(f)$.
From this we see that for every $\lambda \in \C^d$, we have
\begin{equation}\label{eq:Hesse2}
a_{11}\cdot (\lambda^TA_D\ol{\lambda})-
(\lambda^TA_De_1)(\ol{\lambda^TA_De_1})\;\in \;(f).
\end{equation}
Hence $\lambda^T A_D \ol{\lambda}$ is a curve of real contact
with real-contact divisor $f.(\lambda^TA_De_1)-D$.
\end{proof}

In fact, if $M$ is a Hermitian determinantal representation of $f$ and $D$ is the divisor 
$D_M$ defined in Proposition~\ref{prop:divM}, then the systems of curves $\sC(M)$ 
and $\sC_D$ of Definition~\ref{def:CM}~and Lemma~\ref{lem:CD} are the same.  
We can tell whether or not $D$ could come from a determinantal representation
by examining the polynomials in $\sC_D$.

\begin{Prop} \label{prop:exceptional}
  Suppose there exists a real line $\sV_{\C}(\ell)$ which meets $\sV_{\C}(f)$
  in $d$ distinct real points. If $D$ is a real-contact divisor of
  degree $d(d-1)/2$ of $\sV_{\C}(f)$ and $\det(M_D) \equiv 0$, then there exists
  $g\in\sC_D$ with $\ell^2|g$.
\end{Prop}

\begin{proof} 
Suppose that $\det(M_D)$ is identically zero. 
From the proof of Theorem~\ref{thm:DetRepFromDixonProcess}(a), we see that 
$\det(A_D)$ is zero as well. 
First we show that there is some polynomial $g\in \sC_D$ divisible by $\ell$. 
Let $f.\ell=P_1+\cdots+P_d$. 
By construction of $A_D$, we have $\rank(A_D(P_j))\leq1$ for
  $j=1,\dots,d$, so the left kernel of each $A_D(P_j)$ has  dimension at least $d-1$. 
  Thus for each $j$, there is  a nonzero vector $\lambda_j$ contained in the intersection of the 
 left kernels of the $d-1$ matrices $\{A_D(P_k):k\neq j\}$. Since $A_D$ is Hermitian, 
 we see that $\ol{\lambda_j}$ is then in the right kernel of the matrix 
 $A_D(P_k)$ for all  $k\neq j$.

Let $\Lambda$ be the matrix $(\lambda_1 \hdots \lambda_d)$. 
Since $\det(A_D)\equiv 0$, we know that the determinant of the matrix $\Lambda^T A_D \overline{\Lambda}$ is
identically zero. 
Moreover, its off diagonal entries $\lambda_j^T A_D \overline{\lambda_k}$ for
 $j \neq k$ vanish at each of the points  $P_1, \hdots, P_d$. Because these entries 
 have degree $d-1$, they must vanish on the entire line $\sV_{\C}(\ell)$.  
 So modulo the ideal $(\ell)$ the matrix $\Lambda^T A_D \overline{\Lambda}$ 
 is diagonal. Because this matrix has determinant zero, we see that 
 $\ell$ must divide one of the diagonal entries, $\lambda_j^TA_D \ol{\lambda_j}$, 
 which is an element of  $\sC_D$.
 
 Now we claim that this element $g=\lambda_j^TA_D \ol{\lambda_j}$ must be divisible
 by $\ell^2$. We know that $g=\ell h$ for some $h\in \R[x,y,z]$. 
 Since $\sV_{\C}(g)$ is a curve of real contact of $\sV_{\C}(f)$ and the
 intersection multiplicity of $\ell$ and $f$ in each $P_j$ is equal to
  $1$, we must have $h(P_j)=0$ for $j=1,\dots,d$, so that $\ell$ divides $h$ and
  $\ell^2$ divides $g$.
\end{proof}

Using this characterization, we see that if the input divisor $D$ to Construction~\ref{constr:Dixon} 
comes from a curve interlacing $f$, then the resulting matrix $M_D$ is indeed a determinantal representation of $f$.

\begin{proof}[Proof of Theorem~\ref{thm:DetRepFromDixonProcess}(b)]
Suppose $a_{11}= e_1^TA_De_1$ interlaces $f$ with respect to $e$. 
Equation \eqref{eq:Hesse2} shows that for every $\lambda \in \C^d$,
the product $a_{11}\cdot(\lambda^TA_D\ol{\lambda})$ is non-negative 
on $\mathcal{V}_{\R}(f)$. Then, by Lemma~\ref{lem:nonnegInterlacers}, 
we see that $(\lambda^TA_D\ol{\lambda})$ interlaces $f$.
Since $\sV_{\C}(f)$ is smooth, $f$ is  square-free and every polynomial interlacing 
it must also be square-free. 
Hence every polynomial in $\sC_D$ is square-free. 
Since $f$ is hyperbolic, it satisfies the hypothesis of 
Proposition~\ref{prop:exceptional}, and thus
$\det(M_D)$ cannot be zero.

Now we prove that $M_D(e)$ is definite. To do this, we show that $A_D$ is 
the adjugate matrix of $M_D$. By construction, $M_D= f^{2-d}\cdot A_D^{\rm adj}$. 
Taking adjugates, we see that  
\[M_D^{\rm adj} \;\;=\;\; \frac{1}{f^{(d-2)(d-1)}}\cdot (A_D^{\rm adj})^{\rm adj}
\;\;=\;\;  \frac{1}{f^{(d-2)(d-1)}}\cdot\det(A_D)^{d-2}\cdot A_D
\;\;=\;\;c^{d-2}A_D,
\]
where $\det(A_D) = c f^{d-1}$ as in the proof of Theorem~\ref{thm:DetRepFromDixonProcess}(a).
Thus $a_{11}$ is a constant multiple of $e_1^TM_D^{\rm adj}e_1$ and belongs to $\sC(M_D)$. 
Since $a_{11}$ interlaces $f$ with respect to $e$, Theorem~\ref{thm:interlaceDetRep} 
implies that the matrix $M_D(e)$ is definite. 
\end{proof}

\begin{Cor} \label{cor:main}
  Every hyperbolic plane curve possesses a definite Hermitian
  determinantal representation.
\end{Cor}

\begin{proof}
Suppose $f\in \R[x,y,z]_d$ is hyperbolic with respect to $e$ and 
$\sV_{\C}(f)$ is smooth. Then the polynomial
$D_ef$ of \eqref{eq:RenDeriv} interlaces $f$. 
By Proposition~\ref{prop:InterlacingRealContact}, $D_ef$ is a curve of real 
contact to $\sV_{\C}(f)$. Thus, by Theorem~\ref{thm:DetRepFromDixonProcess}, 
using any real-contact divisor coming from $D_ef$ as the 
input for Construction~\ref{constr:Dixon} will result in a 
definite determinantal representation of $f$. 

  Now all that remains is to address singular hyperbolic curves.  Let
  $f\in\R[x,y,z]_d$ be hyperbolic with respect to $e$ with
  $f(e)=1$. By Nuij \cite{Nuij}, there exists a sequence of
  polynomials $(f_k)\subset\R[x,y,z]_d$ converging to $f$ such that for all
  $k$, $f_k$ is hyperbolic with respect to $e$, $f_k(e)=1$, and
  $\sV_{\C}(f_k)$ is smooth. Now each $f_k$ has a Hermitian 
  determinantal representation, hence so does $f$ by Lemma \ref{lemma:DefDetClosed}.
\end{proof}

\begin{Remark}
  One can analyze the relation between real-contact divisors and
  Hermitian determinantal representations more precisely than we have
  done here: If
  $f=\det(M)$ is a Hermitian determinantal representation with
  corresponding real-contact divisor $D$, then the matrix $M_D$
  is Hermite-equivalent to $M$, which means that there exists
  $U\in\GL_d(\C)$ such that $M=U^TM_D\ol{U}$. Furthermore, if two
  Hermitian determinantal representations are Hermite-equivalent, the
  associated real-contact divisors are linearly
  equivalent. Conversely, if $D$ and $D'$ are two linearly equivalent
  real-contact divisors, then $M_D$ is Hermite-equivalent to either
  $M_{D'}$ or $-M_{D'}$.
  For a more detailed discussion, see \cite[Thm.~8]{MR1024486}.
\end{Remark}

\noindent Finally, let us see Dixon's construction in action.

\begin{Example}\label{ex:bigOne}
Here we apply Construction~\ref{constr:Dixon} to the quartic 
\begin{equation}\label{eq:exQuartic}
f(x,y,z) \;\;=\;\; x^4 - 4x^2y^2 + y^4 - 4x^2z^2 - 2y^2z^2 + z^4,
\end{equation}
which is hyperbolic with respect to the point $e=[1:0:0]$.
This curve has two nodes, $[0:1:1]$ and $[0:-1:1]$,
but Dixon's construction will still work. Figure~\ref{fig:Lemma} and Figure~\ref{fig:exampleQuartic} 
show the real curve in the planes $\{z=1\}$ and $\{x=1\}$, respectively.

 \begin{figure}[h]
 \includegraphics[width=1.8in]{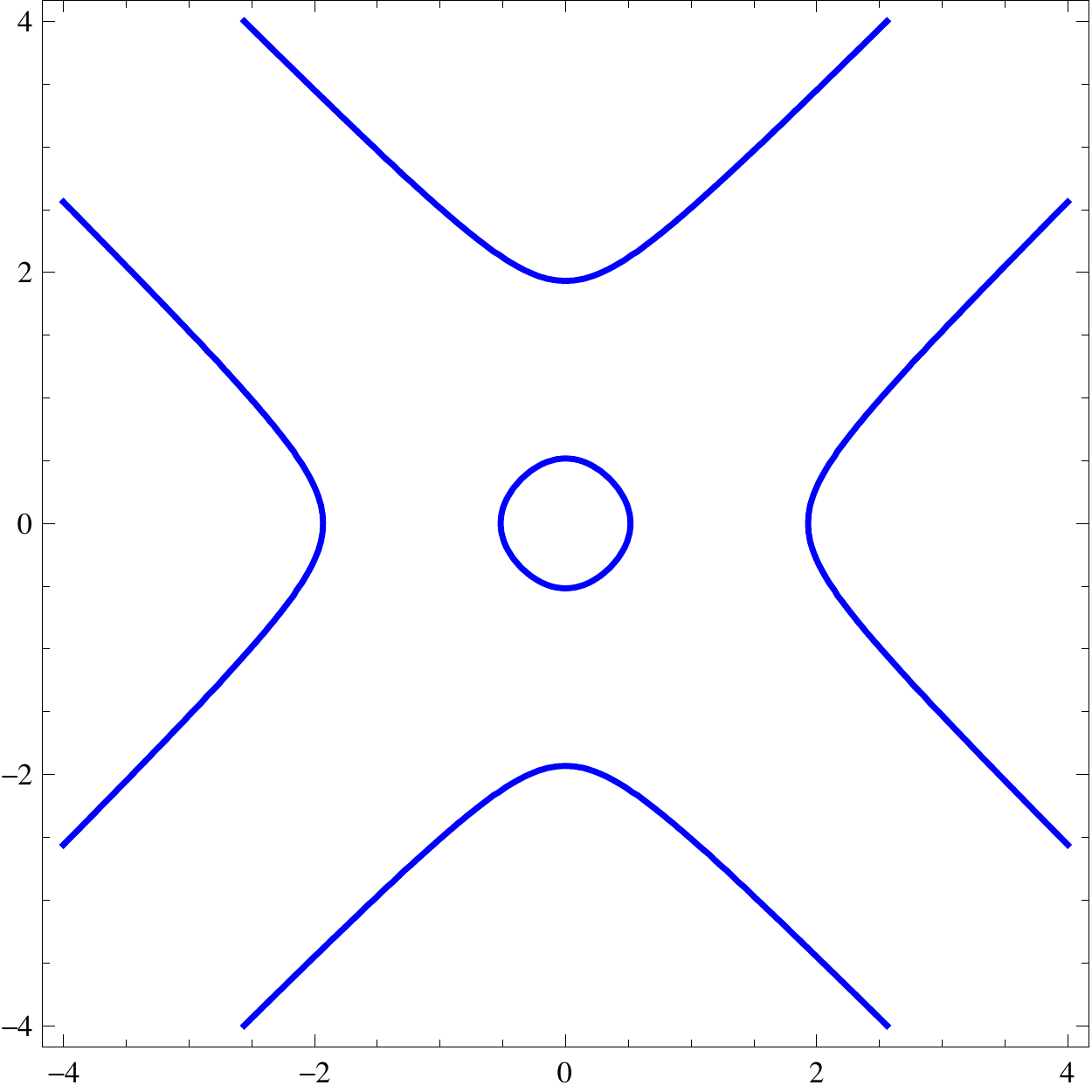} \quad \quad \quad
  \includegraphics[width=1.8in]{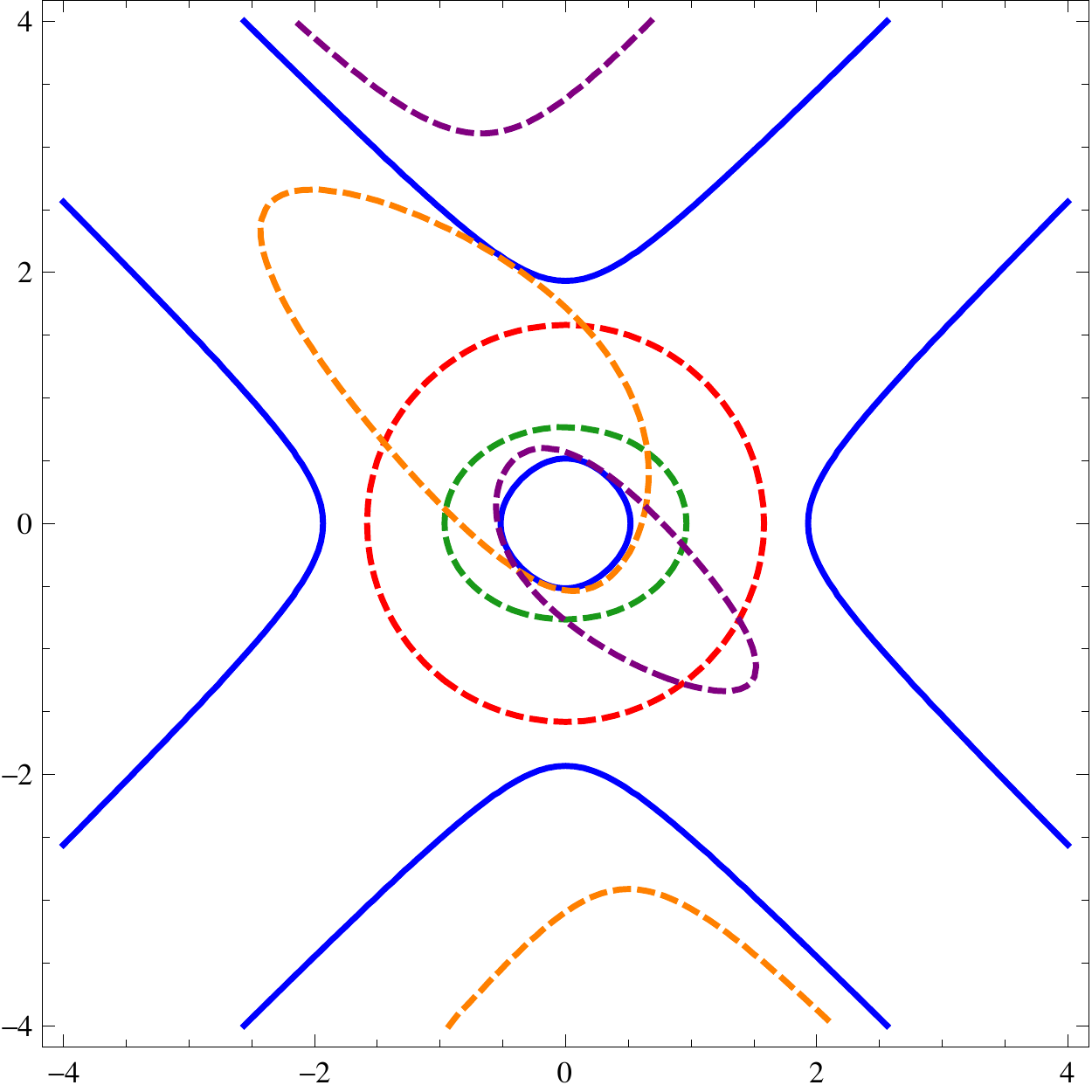} 
\caption{The hyperbolic quartic \eqref{eq:exQuartic} and interlacing cubics from $\sC_D$. }
\label{fig:exampleQuartic}
\end{figure}

First we define $a_{11}$ to be the directional derivative 
$\frac{1}{4} D_ef = x^3 - 2 x y^2 - 2 x z^2$. 
The intersection divisor of $f$ and $a_{11}$ is the sum of the eight points 
$[2:\pm\sqrt{3}:\pm i]$, $[2:\pm i:\pm\sqrt{3}]$ and the two nodes, $[0:\pm1:1]$, 
each with multiplicity 2.  By making some arbitrary choices, we can divide 
these points into two conjugate sets and write the divisor $f.a_{11}$ as
$D+\overline{D}$ where
\[ D  =  [0:1:1] +  [0:-1:1] +[2:\sqrt{3}:i]+[2:-\sqrt{3}:i]+[2:i:\sqrt{3}]+[2:i:-\sqrt{3}].\]
The vector space of cubics in $\C[x,y,z]$ vanishing on these six points 
is four dimensional and we extend $a_{11}$ to a basis 
$\{a_{11},a_{12},a_{13}, a_{14}\}$ for this space, where\
\begin{align*}
a_{12} &= i x^3 + 4 i x y^2 - 4 x^2 z - 4 y^2 z + 4 z^3,\\ 
a_{13} &= -3 i x^3 + 4 x^2 y +  4 i x y^2 - 4 y^3 + 4 y z^2,\\
a_{14} &=-x^3 - 2 i x^2 y - 2 i x^2 z +  4 x y z.
\end{align*}
Then, to find $a_{22}$ for example, we write $a_{12}\cdot \overline{a_{12}}$ as an element
of the ideal $(f,a_{11})$, 
\[a_{12}\cdot \overline{a_{12}} \;=\;(13 x^3 - 14 x y^2 - 22 x z^2)\cdot a_{11}+(16 z^2-12 x^2)\cdot f,\]
and set $a_{22}=13 x^3 - 14 x y^2 - 22 x z^2$.
Similarly  for other $2\leq j \leq k \leq 4$, we find $a_{jk}$ by 
writing $a_{ik}\cdot \ol{a_{1j}}$ as an element of $(f,a_{11})$.  
%\begin{align*}
%a_{33}&=21 x^3 - 30 x y^2 - 6 x z^2,\;\;\;\;\;\;\;\;\;\;\;\;\;\;\;\;
%a_{44} = 5 x^3 - 2 x y^2 - 2 x z^2, \\ 
%a_{23}&=-7 x^3 - 12 i x^2 y - 6 x y^2 -  4 i y^3 + 20 i x^2 z - 32 x y z + 4 i y^2 z + 2 x z^2 +  4 i y z^2 - 4 i z^3, \\
%a_{24} &= 5 i x^3 - 10 x^2 y - 2 i x y^2 + 4 y^3 + 2 x^2 z + 4 i x y z +  2 i x z^2 - 4 y z^2,\\
%a_{34}&= -7 i x^3 + 2 x^2 y - 2 i x y^2 +  14 x^2 z + 4 i x y z + 4 y^2 z + 2 i x z^2 - 4 z^3.
%\end{align*}
The output of Construction~\ref{constr:Dixon} is then the Hermitian matrix of cubics
\[ A_{D} = \begin{pmatrix} a_{11} & a_{12} & a_{13} & a_{14} \\ 
\overline{a_{12}} & a_{22} & a_{23} & a_{24} \\
\overline{a_{13}} & \overline{a_{23}}& a_{33} & a_{34}\\
\overline{a_{14}} &\overline{a_{24}} &\overline{a_{34}} &a_{44}
 \end{pmatrix}.\]
 By taking the adjugate of $A_D$ and dividing by $f^2$, we find the desired 
 Hermitian determinantal representation,
\[ M_{D} \;\;=\;\; \frac{1}{f^2} \cdot A_{D}^{\rm adj} \;\;=\;\; 2^5 
 \begin{pmatrix} 
14 x& 2 z& 2 i x - 2 y& 2 i (y - z)\\ 
2 z& x&   0& -i x + 2 y\\ -2 i x - 2 y& 0& x& i x - 2 z\\ 
-2 i (y - z)&   i x + 2 y& -i x - 2 z& 4 x\end{pmatrix}.
 \]
The determinant of $M_D$ is $2^{24}\cdot f$. As promised by Theorems~\ref{thm:interlaceDetRep} and \ref{thm:DetRepFromDixonProcess}, 
the cubics in $\sC_D=\sC(M_D)$ interlace $f$ (see Figure~\ref{fig:exampleQuartic}) 
and the matrix $M_D$ is positive definite
at the point $(x,y,z)=(1,0,0)$.
\end{Example}

In general, the challenge of carrying through Construction~\ref{constr:Dixon} 
in exact arithmetic is the computation of the intersection points. 
By contrast, computing
a symmetric determinantal representation from a given contact curve is
much simpler, but it may be very difficult to find
a suitable curve to start from. For further algorithmic results, especially in the
case of quartics, see \cite{quartics} and \cite{us2}. Numerical computations
seem more promising and we plan to pursue this in a future project.

\section{Hyperbolicity cones and spectrahedra}\label{sec:spect}

For a polynomial $f$ that is hyperbolic with respect to $e\in \R^{n+1}$, 
the connected component of $e$ in the complement of the hypersurface $\sV_{\R}(f)$ 
plays a special role. This is a \emph{hyperbolicity cone}, denoted $C(f,e)$ 
and can also be defined as
\[ C(f,e) = \{ a\in \R^{n+1} \;:\; f(te-a) \neq 0 \;\text{ when }\; t\leq 0\}.\]
As shown in G\r{a}rding \cite{MR0113978}, $C(f,e)$ is a convex cone and $f$ is
hyperbolic with respect to any point contained in it. 

A \emph {hyperbolic program} is the problem of optimizing a linear function 
over an affine slice of a hyperbolicity cone. Hyperbolic programming
is a generalization of semidefinite programming, the problem of optimizing a linear function over an affine slice of 
the cone of positive semidefinite symmetric matrices. Such convex bodies are called \emph{spectrahedra}.
Because the determinant is a hyperbolic polynomial on the space of real symmetric matrices, we see that
every spectrahedral cone is indeed a hyperbolicity cone. A major open question is whether or not the converse holds. 
\medskip

\noindent {\bf Generalized Lax Conjecture.} Every hyperbolicity cone is a spectrahedron.
\medskip

Showing that a hyperbolicity cone $C(f,e)$ is spectrahedral amounts to
finding a definite real symmetric determinantal representation for $f$
(or for an appropriate multiple of $f$).  For a detailed discussion, see
\cite[Conjecture 3.3]{Vin}. The work of Helton-Vinnikov \cite{HV}
settled this for three dimensional hyperbolicity cones by showing that
every hyperbolic polynomial in three variables has a definite symmetric
determinantal representation. We conclude by noting that one can obtain the same result
from the existence of definite Hermitian determinantal
representations.
  
\begin{Cor}
  Every three-dimensional hyperbolicity cone is a spectrahedron.\end{Cor}

\begin{proof}
  Let $f\in\R[x,y,z]_d$ be hyperbolic with respect to $e$. By 
  Corollary~\ref{cor:main}, $f$ admits a definite Hermitian determinantal
  representation $f=\det(M)$. We can write $M=A+iB$, where $A$ is real
  symmetric and $B$ is real skew-symmetric, and define $N$ to be the real symmetric matrix 
 \[N \;\;=\;\; \begin{bmatrix}
     A & B\\
     -B & A
  \end{bmatrix}.\]
By the change of coordinates,
\[U^TN\overline{U} =   
\begin{bmatrix}
     A-iB & 0\\
     0 & A+iB
  \end{bmatrix}
\;\;\;\;\; \text{where} \;\;\;\; 
U = 
\begin{bmatrix}
     \frac{1}{\sqrt{2}}\cdot I & \frac{i}{\sqrt{2}}\cdot I \\
      \frac{i}{\sqrt{2}}\cdot I  &  \frac{1}{\sqrt{2}}\cdot I 
  \end{bmatrix},
\]
  we see that $\det(N)=
  \det(\overline{M})\det(M) = f^2$.
The hyperbolicity cone of $f$ is the same as that of $f^2$
(with respect to $e$), which is the spectrahedron
described by $N$.
\end{proof}

\end{document}